\documentclass[12pt]{amsart}
\usepackage[utf8]{inputenc}
\usepackage{amssymb, amsmath} 
\usepackage{amsfonts}
\usepackage{amssymb}
\usepackage{amsthm}
\usepackage{caption}
\usepackage{subcaption}
\usepackage{graphicx}
\usepackage{tikz}
\usepackage{color}
\usepackage{calc}
\usetikzlibrary{arrows,positioning,automata,shapes,calc,3d,graphs}
\usetikzlibrary{decorations.markings}
\usetikzlibrary{arrows, shadows}
\usepackage[underline=true,rounded corners=false]{pgf-umlsd}
\usepackage{xspace}

\newcommand{\bm}[1]{\mbox{\boldmath $#1$}}

\newcommand{\cds}{{\upshape{\textbf{cds}}}\xspace} 
\newcommand{\cdr}{{\upshape{\textbf{cdr}}}\xspace} 

\newcommand{\fixcdr}{\bm{\operatorname{cdr^{\circ}_{\mathit{n}}}\xspace}} 
\newcommand{\fixcds}{\bm{\operatorname{cds^{\circ}_n}\xspace}} 

\newcommand{\s}{\mathfrak{S}} 

\theoremstyle{plain}
\newtheorem{theorem}{Theorem}[section]
\newtheorem{lemma}[theorem]{Lemma}
\newtheorem{conj}[theorem]{Conjecture}
\newtheorem{corollary}[theorem]{Corollary}

\newtheorem{proposition}[theorem]{Proposition}

\newtheorem*{repp@theorem}{\repp@title (reformulated)}
\newcommand{\newrepptheorem}[2]{%
\newenvironment{repp#1}[1]{%
 \def\repp@title{#2 \ref{##1}}%
 \begin{repp@theorem}}%
 {\end{repp@theorem}}}
\makeatother
\newrepptheorem{theorem}{Theorem}

\theoremstyle{definition}
\newtheorem{definition}[theorem]{Definition}
\newtheorem{example}[theorem]{Example}
\makeatletter
\newtheorem*{rep@theorem}{\rep@title \ continued}
\newcommand{\newreptheorem}[2]{%
\newenvironment{rep#1}[1]{%
 \def\rep@title{#2 \ref{##1}}%
 \begin{rep@theorem}}%
 {\end{rep@theorem}}}
\makeatother
\newreptheorem{example}{Example}

\tikzstyle{vertex}=[circle, draw, inner sep=0pt, minimum size=10pt]
 \usepackage{graphicx, scalefnt}
 \usepackage{pgfplots}
 \usepackage{etex}
 \usetikzlibrary{decorations.pathmorphing}
\usetikzlibrary{shapes,arrows}
\tikzstyle{decision} = [diamond, draw, text width=5em, text badly centered, node distance=3cm, inner sep=0pt]
\tikzstyle{block} = [rectangle, draw, text width=5em, text centered, rounded corners, minimum height=4em]

\makeatletter
\@namedef{subjclassname@2010}{%
  \textup{2010} Mathematics Subject Classification}
\makeatother

\title{Sorting permutations: games, genomes, and cycles}

\author[Adamyk]{K.L.M. Adamyk}
\address{Department of Mathematics, University of Colorado, Boulder, CO}
\email{katharine.adamyk@colorado.edu}

\author[Holmes]{E. Holmes}
\address{Department of Mathematics, University of Hawaii at Manoa, Honolulu, HI}
\email{eholmes@math.hawaii.edu}

\author[Mayfield]{G.R. Mayfield}
\address{Mathematics Department, Willamette University, Salem, OR}
\email{gmayfiel@willamette.edu}

\author[Moritz]{D.J. Moritz}
\address{Department of Mathematical Sciences, University of Montana, Missoula, MT}
\email{dennis.moritz@umconnect.umt.edu}

\author[Scheepers]{M. Scheepers}
\address{Department of Mathematics, Boise State University, Boise, ID} 
\email{mscheepe@boisestate.edu}

\author[Tenner]{B.E. Tenner$^\dagger$}
\address{Department of Mathematical Sciences, DePaul University, Chicago, IL}
\email{bridget@math.depaul.edu}
\thanks{$^\dagger$ Research partially supported by a Simons Foundation Collaboration Grant for Mathematicians.}

\author[Wauck]{H.C. Wauck}
\address{Department of Computer Science, University of Illinois at Urbana-Champaign, Urbana,~IL}
\email{hwauck@gmail.com}


\begin{document}

\begin{abstract}
Permutation sorting, one of the fundamental steps in pre-processing data for the efficient application of other algorithms, has a long history in mathematical research literature and has numerous applications.  Two special-purpose sorting operations are considered in this paper: {\tt context directed swap}, abbreviated \cds, and {\tt context directed reversal}, abbreviated \cdr. These are special cases of sorting operations that were studied in prior work on permutation sorting. Moreover, \cds and \cdr have been postulated to model molecular sorting events that occur  in the  genome maintenance program of certain species of single-celled organisms called ciliates. 

This paper investigates mathematical aspects of these two sorting operations. The main result of this paper is a generalization of previously discovered characterizations of \cds-sortability of a permutation. The combinatorial structure underlying this generalization suggests natural combinatorial two-player games. These games are the main mathematical innovation of this paper.
\end{abstract}

\maketitle

A \emph{permutation} of $\{1,\ldots,n\}$ is a repetition-free list of the first $n$ positive integers. A procedure that rewrites this list in increasing order has \emph{sorted} the permutation. If a specified operation successfully sorts a permutation, then we say that the permutation is \emph{sortable} by that operation. Definitions of other technical terms used in this introduction will be given in later sections of the paper.

Due to its utility in streamlining numerous search algorithms in everyday use, permutation sorting has been extensively studied. Of the significant body of work on permutation sorting, we review only the sorting operations most directly connected with the focus of this paper.  These operations include reversals and transpositions. 

Much work on the efficiency of sorting permutations by reversals were motivated by the arrival of data from genome sequencing. There are now numerous examples of pairs of organisms A and B for which the relative positions of gene locations on a chromosome of organism A is a permutation of the relative positions of the corresponding genes on a chromosome of organism B (see, for example, \cite{HP, Schetal}).  In the 1930s, biologists Dobzhansky and Sturtevant proposed using the number of inversions (also called ``reversals'') required to sort the gene order of organism A to that of organism B as a criterion for measuring the evolutionary distance between these organisms \cite{DS}. This raised the problem of determining the minimum number of reversals required to sort a given permutation (the \emph{reversal distance problem}), and the corresponding problem of efficiently finding a minimal length sequence of reversals that would accomplish the sorting. Pevzner and collaborators achieved significant results in determining the reversal distance, and in determining a sequence of reversals that accomplishes the sorting in the minimal number of sorting steps. In particular, in the case of signed permutations (permutations where some values are flagged as ``negative''), efficient algorithms for determining the reversal distance and for finding a minimum length sequence of reversals to sort the signed permutation have been established in \cite{HP}. In that work a special type of reversal called an \emph{oriented reversal} was identified as playing a crucial role in minimizing the number of reversal sorting steps. 

In our paper, an oriented reversal is a \emph{context directed reversal}, abbreviated \cdr. There are good reasons for this terminology change. In an independently developing investigation D.M. Prescott discovered that in certain ciliate species, single-celled organisms with the feature of harboring two types of nuclei, micronuclei and macronuclei, the genes appearing in the macronucleus appear in encrypted form in the micronucleus. In particular, several segments of a gene in the macronucleus may appear embedded in separate locations on chromosomes in the micronucleus, in permuted order and sometimes opposite orientation, from the appearance of these segments in the macronuclear gene. Thus the micronuclear version of the gene can be represented as a permutation or a signed permutation  of these segments of the macronuclear gene. In addition, during certain events in the cell cycle of these ciliates, the macronucleus is discarded, while a new macronucleus is constructed by decrypting a copy of the micronucleus. This decryption process involves sorting micronuclear segments of genes into their canonical macronuclear versions. A good introduction to this interesting phenomenon can be found in the reviews \cite{DMP1, DMP}. Two biomolecular models for this sorting process in ciliates have been proposed. The one relevant to the topic of this paper is given in \cite{EPModel}. In the corresponding mathematical representation of the postulated sorting operations of this model, one of the sorting operations is denoted \textsf{hi} (for \emph{hairpin inverted repeat}). It turns out that the \textsf{hi} sorting operation is the \cdr operation. Several questions arise about this postulated operation, including determining which signed permutations are sortable by the \cdr sorting operation. A characterization of signed permutations sortable by \cdr can be gleaned from \cite[Section 5]{HP}. It was also independently obtained in \cite{EHPR}, and again later using a different mathematical framework in \cite{BHR}. 

A second sorting operation on permutations, \emph{block interchanges}, was investigated by Christie in \cite{DC}. In that work Christie proved that there is an efficient method to determine for given permutations A and B the minimum number of block interchanges required to sort A to B. This is the \emph{block interchange distance problem}. In \cite[Lemma 2]{DC} a special kind of block interchange, called a \emph{minimal block interchange} by Christie, emerged as central to minimizing the number of block interchanges required to sort a permutation. Some instances of this specialized block interchange are examples of what we call \emph{context directed swaps}, abbreviated \cds.

Specialized block interchange operations at a molecular level were also independently and later postulated in \cite{EPModel} to be one of the biomolecular sorting operations active during ciliate micronuclear decryption. The corresponding mathematical representation of that sorting operation is the sorting operation \cds of this paper. In studies of the mathematical models of ciliate micronuclear decryption, \cds has been called \textsf{dlad} - an abbreviation for \emph{double loop alternating direct repeat}. The name \textsf{dlad}, like \textsf{hi}, was inspired by the geometrical configurations described in the biomolecular model of \cite{EPModel}. Again the first concern was to identify which permutations are \cds-sortable. As with \cdr, the permutations sortable by \cds were characterized in \cite{EHPR} and later again in \cite{BHR} using a different mathematical framework.

Readers interested in these two lines of mathematical investigation related to biological phenomena might appreciate two recent textbooks. A thorough survey of combinatorial as well as algorithmic aspects of work related to genome rearrangements appears in \cite{FLRTV}, while \cite{EHPPR} gives an in-depth introduction to the micronuclear decryption phenomenon in ciliates, as well as an extensive coherent treatment of mathematical modeling of the process and findings from several prior research articles.

While sortability criteria and efficient algorithms to determine the number of sorting steps in a successful sorting are known for each of \cds and \cdr,  several other mathematical aspects of these sorting operations on (signed) permutations remain to be investigated. 
Sorting steps in applications of \cdr and \cds are irreversible. After a finite number of applications of one of these sorting operations, a state is reached where no further sorting steps are possible. These states are called \emph{fixed points} of the corresponding sorting operation. Thus, a (signed) permutation is sorted when the fixed point reached is the identity permutation. In this paper the \cds-sortability criterion is generalized to provide a linear time criterion for determining, for an arbitrary permutation, which \cds fixed points are achievable by applications of \cds. The \cds-sortability criterion from earlier works is the special case when the fixed point in question is the identity permutation.

For some (signed) permutations the fixed point reached through applications of the featured sorting operation depends on the order in which the legal sorting steps are executed. This strategic aspect of sorting is of independent interest and can be examined by means of two-player combinatorial games. By a classical theorem of Zermelo \cite{Z}, for each instance of the combinatorial games introduced in this paper one of the players has a winning strategy. This raises a fundamental decision problem: for a given instance of the game, which of the two players has a winning strategy in the game? We obtain preliminary results on this problem for the \cds sorting operation. One of these results has since been shown in other work \cite{JSST} to be optimal.

With sortability criteria known for (signed) permutations, progress on enumerative work would seem to be in reach. However, only limited information is available on even such basic questions as how many permutations of $n$ symbols are \cds-sortable. Preliminary findings on this problem are also reported in this paper. 

For this paper we had several choices among the frameworks developed in the cited publications to use to present our results, including various types of graphs, algebras over strings, permutations, and so on. For the study of \cds, we elected to use an associated permutation representation developed by \cite{DL}. We did this for several reasons: (i) numerous software packages for use in combinatorial experimentation have specific modules developed for using permutations - including MAPLE, Sage, etc., (ii) students of mathematics encounter the study of permutations and the familiar notions used in this paper in their undergraduate curriculum, and (iii) several known enumerative results regarding permutations are directly available for use in this investigation. For the study of \cdr, we also use a suitable permutation representation, but ours is somewhat different from the one used in \cite{DL}.

In Section~\ref{section:definitions} of this paper, we introduce fundamental concepts underlying much of our study, namely pointers and signed permutations. In Section~\ref{section:cds}, we briefly review sorting by transpositions and by block interchanges that was studied in the past by Christie \cite{DC}, among others. Then we introduce \emph{context directed swaps}, denoted \cds, which are a restricted version of the block interchange operation. In Theorem~\ref{strategicpilecorollary}, using an object called \emph{the strategic pile} of a permutation, we characterize for each permutation which \cds fixed points are obtainable by applications of \cds.  In Section~\ref{sec:cds implications}, we explore results and objects related to the \cds-sorting operation following from our results in the previous section. Perhaps the most interesting from the point of view of modeling ciliate micronuclear decryption are the \cds Inevitability Theorem, Theorem \ref{cdsinevitable}, and the \cds Duration Theorem, Theorem \ref{cdssteps}. 

In Section~\ref{section:cdsgames} we define two-person games by exploiting the fact that the \cds fixed point of a permutation reached by applications of \cds varies according to the order in which \cds operations are applied. By a classical theorem of Zermelo \cite{Z}, in each of the games we introduce one of the players has a winning strategy. In Theorem~\ref{cdsgreedy} we identify a class of permutations for which the first player to move has a winning strategy in permutation sorting games based on \cds, and a class of permutations for which the second player to move has a winning strategy in permutation sorting games based on \cds. Subsequent work in \cite{JSST}, establishing that Theorem~\ref{cdsgreedy} is to an extent optimal, suggests that the decision problem of determining for a generic permutation which player has the winning strategy may be of high complexity. 

In Section~\ref{section:cdr}, we introduce the sorting operation \cdr and a number of basic decision problems related to \cdr, and discuss related games in Section~\ref{section:cdrgames}. Although the main focus of this paper is the mathematical theory of sorting permutations by constrained sorting operations, we do briefly discuss some of the broader implications of some of our results in Section~\ref{section:application}. Finally, in Section~\ref{section:future} we point out a few specific mathematical problems emerging from this work, and we point out problems raised by some recent experimental findings in connection with the ciliate decryptome.  

\section{Permutations, signed permutations, and pointers}\label{section:definitions}
 
For a positive integer $n$, the set $\s_n$ consists of all permutations of $\{1,\ldots,n\}$. This is the \emph{symmetric group}, also known as the \emph{finite Coxeter group of type $A$}. There are many common ways to denote permutations in the literature. For the most part, we will use \emph{inverse image notation}, in which $\pi \in \s_n$ would be written as
$$\pi = [\pi^{-1}(1) \ \pi^{-1}(2) \ \cdots \ \pi^{-1}(n)].$$
In Section~\ref{section:cds}, we will also need to use \emph{cycle notation}, in which a permutation is written as a product of disjoint cycles of the form $(a \ \pi(a) \ \pi^2(a) \ \cdots)$, usually oriented so that $a$ is the minimum value of its cycle.

\begin{example}
The permutation $\pi \in \s_8$ whose inverse image notation is
$$\pi = [2 \ 7 \ 1 \ 5 \ 8 \ 6 \ 3 \ 4]$$
can be written in cycle notation as
$$\pi = (1372)(485)(6).$$
\end{example}

The \emph{finite Coxeter group of type $B$} consists of \emph{signed} permutations, and we will denote these sets by $\s^{\pm}_n$. Signed permutations are bijections $\pi$ on $\{\pm1,\ldots,\pm n\}$ that satisfy $\pi(-i) = -\pi(i)$ for all $i$. Note that this requirement means that $|\s^{\pm}_n| = 2^nn!$, and that a signed permutation $\pi \in \s^{\pm}_n$ can be completely described by the inverse image notation
$$\pi = [\pi^{-1}(1) \ \pi^{-1}(2) \ \cdots \ \pi^{-1}(n)].$$
Thus we can say that such a $\pi$ is a signed permutation \emph{of $n$ letters}.
 
The \cds and \cdr sorting operations will be defined in terms of ``pointers'' in a string. Consider a (signed or unsigned) permutation $\pi = [ a_1 \ \cdots \ a_n]$, and recall that $|a_i| \in \{1,\ldots,n\}$ for each $i$. To each $a_i \not\in\{\pm 1,\pm n\}$ we associate two \emph{pointers}, while $a_i \in \{\pm 1\}$ gets only a right pointer and $a_i \in \{\pm n\}$ gets only a left pointer. The \emph{left} pointer is
$$\lambda(a_i) = \begin{cases}
(a_i - 1, a_i)& \text{if } a_i  >1, \text{ and}\\
-(|a_i| + 1, |a_i|) & \text{if } a_i< 0,
\end{cases}$$
and the \emph{right} pointer is
$$\rho(a_i) = \begin{cases}
(a_i , a_i+1) & \text{if } 0 < a_i < n, \text{ and} \\
-(|a_i| , |a_i| - 1) & \text{if } a_i< 0.
\end{cases}$$
The entries of a pointer are always positive integers. We shall call a pointer $-(i+1,i)$ the \emph{negative} of the pointer $(i,i+1)$. We write $q = -p$ to denote this relationship between $p$ and its negative, $q$. Note that each pointer appears twice among the $2n-2$ pointers in a permutation of $\s_n$ or (possibly in its negative version) in a signed permutation of  $\s^{\pm}_n$.

\begin{example}
In the signed permutation $\pi = [5 \ -3 \ 2 \ -4 \ 1]$, the entry $5$ has left pointer $(4,5)$, and no right pointer. Similarly the entry $-3$  has left pointer $-(4,3)$ and right pointer $-(3,2)$. Here is $\pi$ with all pointers marked:
$$\pi =[_{(4,5)}5 \ \ \ _{-(4,3)}-\!3_{-(3,2)} \ \ \ _{(1,2)}2_{(2,3)} \ \ \ _{-(5,4)}-\!4_{-(4,3)} \ \ \ 1_{(1,2)} ].$$
\end{example}

\begin{definition}
Fix a (signed or unsigned) permutation $\pi = [a_1 \ \cdots \ a_n]$, and $i < n$. We call $a_i$ an \emph{adjacency} of $\pi$ if $a_i + 1 = a_{i+1}$. 
\end{definition}

Note that an unsigned permutation is sorted if every position is an adjacency.

\section{Sorting by context directed block swaps, and the strategic pile}\label{section:cds}

In this section we consider unsigned permutations. These are the elements of $\s_n$. 
\begin{definition}
Consider a permutation $\pi$ in which the pointers $p$ and $q$  appear in the order $p \cdots q \cdots p \cdots q$. The \emph{\cds operation on $\pi$ with context $\{p,q\}$} swaps the two blocks of letters that are flanked on the left by $p$ and on the right by $q$.
\end{definition}

For $p = (x,x+1)$ and $q = (y,y+1)$, the \cds operation on a permutation $\pi$ with context $\{p,q\}$ has one of the forms depicted in Table~\ref{table:cds}. 

\begin{table}[htbp]
$$\begin{array}{|l|l|}\hline
& \text{Result of \cds with context }\\
\text{Permutation } \pi & \{(x,x+1),(y,y+1)\} \\ \hline
[ \cdots \ x \ \alpha \ y \ \beta \ (x+1) \ \gamma \ (y+1) \ \cdots] & [ \cdots \ x \ (x+1) \ \gamma \ \beta \ \alpha \ y \ (y+1) \ \cdots]\\
\hline
[ \cdots \ x \ \alpha \ (y+1) \ \beta \ (x+1) \ \gamma \ y \ \cdots] & [ \cdots \ x \ (x+1) \ \gamma \ y \ (y+1) \ \beta \ \alpha \ \cdots]\\
\hline
[ \cdots \ (x+1) \ \alpha \ y \ \beta \ x \ \gamma \ (y+1) \ \cdots] & [ \cdots \ \gamma \ \beta \ x \ (x+1) \ \alpha \ y \ (y+1) \ \cdots]\\
\hline
[ \cdots \ (x+1) \ \alpha \ (y+1) \ \beta \ x \ \gamma \ y \ \cdots] & [ \cdots \ \gamma \ y \ (y+1) \ \beta \ x \ (x+1) \ \alpha \ \cdots]\\
\hline
\end{array}$$
\caption{\cds operation with context $\{(x,x+1),(y,y+1)\}$ where $x$ and $y$ are symbols, and $\alpha$, $\beta$ and $\gamma$ are substrings of $\pi$.}\label{table:cds}
\end{table}

In \cite{EHPR} the \cds operation is named \textsf{dlad}. In \cite{DC} block interchanges of any two disjoint segments of entries in a permutation are considered. Towards achieving a sorted permutation in a minimal number of sorting steps, \cite{DC} identifies a special type of block interchange named a \emph{minimal block interchange}: These are described as follows:   

For a permutation $\pi$, first ``frame" elements of $\s_n$ by adding the entry ``$0$" at the leftmost end, and ``$n+1$" at the rightmost end, thus: For $\pi = \lbrack \pi_1 \ \pi_2 \ \cdots \ \pi_n\rbrack$, the framed version $\pi^{\star}$ of $\pi$ is $\lbrack 0 \ \pi_1 \ \pi_2 \ \cdots \ \pi_n \ (n+1) \rbrack$. 
If $\pi$ is not the identity permutation then there are numbers $x < y$ with $y$ appearing to the left of $x$ in $\pi$. 
Fix the least such $x$, and then fix the largest $y$ appearing to its left. Then with this information emphasized, the framed version of $\pi$ is
$\pi^{\star} = \lbrack 0 \ \alpha \  (x-1) \ \beta \ y \ \gamma \ x \ \delta \ (y+1) \ \nu \ (n+1) \rbrack,$
where $x$ and $y$ are symbols, and $\alpha$, $\beta$, $\gamma$, $\delta$ and $\nu$ are substrings of $\pi$. There are the following possibilities regarding $x-1$ and $y+1$:
When $x=1$ then $x-1 = 0$ is not a symbol in the original $\pi$, and then $\pi^{\star}$ has the form $\lbrack 0 \ \beta \ y \ \gamma \ 1 \ \delta \cdots \ (n+1)\rbrack$. When $y = n$ then $y+1 = n+1$ is not a symbol in the original $\pi$, and then $\pi^{\star}$ has the form $\lbrack 0 \ \alpha \  (x-1) \ \beta \ n \ \gamma \ x \ \delta \  (n+1) \rbrack$. If $x>1$ and $y<n$ then $x-1$ and $y+1$ are symbols in the original $\pi$. 

Now the \emph{minimal block interchange} associated with $x$ and $y$ swaps the segments $\beta \ y$ and $x \ \delta$, producing $\phi^{\star} = [0 \ \alpha \  (x-1) \ x \ \delta  \ \gamma \ \beta \ y \ (y+1) \ \nu \ (n+1)]$. In the case when $1<x$ and $y<n$, this minimal block interchange is in fact an application of \cds with pointer context $\{ (x-1,\; x), \ (y,\; y+1)\}$. In the cases when $x=1$ or $y=n$, the corresponding minimal block interchange is \emph{not} an application of \cds.

\begin{definition}
The permutation $\varphi$ is a \emph{fixed point} of \cds if there are no pointers $p$ and $q$ that appear in $\varphi$ as $p \cdots q \cdots p \cdots q$. 
\end{definition}

\begin{lemma}\label{lem:cdsfixedpoints}
The \cds fixed points in $\s_n$ are the permutations $[k \ (k+1) \ \cdots \ n \ 1 \ \cdots \ (k-1)]$ for $k\ge 1$. These permutations form the cyclic subgroup of $\s_n$ generated by
$$[2 \ 3 \ \cdots \ n \ 1] = (1 \ n \ n-1 \ \cdots \ 3 \ 2).$$
\end{lemma}

In the case when $\pi\in \s_n$ is a \cds fixed point, say $\pi = \lbrack (k+1) \ (k+2) \ \cdots \ n \ 1 \ 2 \ \cdots \ k\rbrack$, the corresponding minimal block interchange (in the sense of \cite{DC}) produces the identity permutation, and corresponds to the \emph{ boundary }\textsf{ ld} operation used in \cite{EHPR}. Neither of these operations is an example of \cds.  

Towards characterizing for an element $\pi$ of $\s_n$ which of the \cds fixed points are obtainable from applications of \cds to $\pi$, we introduce the following construct, first described in \cite{DL}: 

\begin{definition}\label{defn:permutation product}
Given $\pi \in \s_n$ with inverse image notation $\pi = [a_1 \ a_2 \ \cdots \ a_n]$, define two permutations on the set $\{0,1,\ldots,n\}$, 
written in cycle notation, as follows:
$$X = (0 \ 1 \ 2 \ \cdots \ n)$$ 
and
$$Y_{\pi} = (
a_n \ a_{n-1} \ \cdots \ a_2 \ a_1 \ 0).$$
Now set
$$C_{\pi} = Y_{\pi}X,$$
where products of permutations are compositions of maps, and so multiply from right to left.
\end{definition}

\begin{example}\label{ex:computing C_{pi}}
Let $\pi = [4 \ 2 \ 6 \ 7 \ 1 \ 3 \ 5] \in \s_7$. Then $Y_{\pi} = ( 
5 \ 3 \ 1 \ 7 \ 6 \ 2 \ 4 \ 0)$ and $C_{\pi} = (0 \ 7 \ 5 \ 2 \ 1 \ 4 \ 3)(6)$.
 The value $6$ was an adjacency of $\pi$ because $a_3 = 6$ and $a_4 = 7$, and we see that $C_{\pi}$ does indeed fix the value $6$. 
\end{example}

We begin with some straightforward observations about the permutation $C_{\pi}$ for $\pi=  \lbrack a_1 \ \cdots \ a_n\rbrack \in \s_n$.

\begin{lemma}\footnote{We say that ``$C_{\pi}$ fixes $a_i$".}
If $\pi$ has an adjacency at $a_i$, then the cycle decomposition of $C_{\pi}$ 
includes the $1$-cycle $(a_i)$.  
\end{lemma}

Suppose now that $\pi$ has a pointer pair $p = (x,x+1)$ and $q=(y,y+1)$ appearing as $p \cdots q \cdots p \cdots q$. By \cite[Lemma 2]{DC}, applying \cds to $\pi$ with this context produces a permutation $\sigma$ having exactly two more cycles in its cycle decomposition than $\pi$ has.

\begin{lemma}\label{lem:how cds affects cycles} 
Suppose that $\pi$ has a pointer pair $p = (x,x+1)$ and $q = (y,y+1)$ that appear in alternating order. Let $\pi'$ be the result of applying \cds to $\pi$ with context $\{p,q\}$. Then the only cycles in the cycle decomposition of $C_{\pi}$ that do not also appear in the cycle decomposition of $C_{\pi'}$ are those that contain $x$ or $y$. Moreover, $x$ and $y$ are fixed by the permutation $C_{\pi'}$, and the rest of the cycle decomposition for $C_{\pi'}$ is obtained from that for $C_{\pi}$ by deleting $x$ and $y$ from the cycles in which they appear.
\end{lemma}

\begin{proof}
Use Table~\ref{table:cds} and the definition of $C_{\pi}$.
\end{proof}

\begin{repexample}{ex:computing C_{pi}}
The pointers $(3,4)$ and $(5,6)$ appear in alternating order in $\pi = [4 \ 2 \ 6 \ 7 \ 1 \ 3 \ 5]$. If we apply \cds with that context, we obtain $\sigma = [5 \ 6 \ 7 \ 1 \ 3 \ 4 \ 2]$, for which $C_{\sigma} = (0 \ 7 \ 2 \ 1 \ 4)(3)(5)(6)$. 
\end{repexample}

Observe from Lemma~\ref{lem:how cds affects cycles} that applying \cds to the permutation $\pi$ produces two fixed points in the permutation $C_{\pi}$. An element of $\s_n$ can have at most $n$ fixed points, so there is a limit to the number of times \cds could possibly be applied.

\begin{corollary}[\cds Termination]\label{cdsuniversal}
The sorting operation \cds can be applied only finitely many times to a given permutation before producing a \cds fixed point.
\end{corollary}

Now we consider the following decision problem: 

\begin{quote}
\framebox{
\parbox{.8\textwidth}{
{\bf D.1 \hspace{0.1in} \cds FIXED POINT:}

INSTANCE: A \cds fixed point $\varphi$ and a permutation $\pi \in \s_n$.

QUESTION: Is $\varphi$ a \cds fixed point of $\pi$?}}
\end{quote}

In Example \ref{ex:fixedptdecision} we illustrate the intuitive approach to answering a specific instance of this decision problem.

\begin{example}\label{ex:fixedptdecision} Consider the permutation $\pi =  [4 \ 1 \ 3 \ 2]$
and the \cds fixed point $\varphi = \lbrack 1 \ 2 \ 3 \ 4 \rbrack$. The corresponding instance of the decision problem {\bf D.1} asks whether $\varphi$ is a \cds fixed point of $\pi$. One approach to answering this instance is to determine all the \cds fixed points of $\pi$ by applications of \cds: Applying \cds for the context $\{(1,2),\; (2,3)\}$ to $\pi$ produces the \cds fixed point $\lbrack 4 \ 1 \ 2 \ 3 \rbrack$; applying \cds for the context $\{(1,2),\; (3,4)\}$ produces the \cds fixed point $\lbrack 3 \ 4 \ 1 \ 2 \rbrack$;  applying \cds for the context $\{(2,3),\; (3,4)\}$ produces the \cds fixed point $\lbrack 2 \ 3 \ 4 \ 1 \rbrack$. These are all the \cds fixed points obtainable from applications of \cds to $\pi$. Thus, in this instance of the decision problem {\bf D.1} the answer is ``no", as $\varphi$ is not obtainable from $\pi$ through applications of \cds.
\end{example}

One might ask whether there is a more efficient method for determining an answer to decision problem {\bf D.1}. The special case of this decision problem when $\varphi$ is the identity element of $\s_n$, which we may call the \cds \emph{Sortability} decision problem,  has been given a linear time solution in prior literature \cite{BHR, EHPR}. In this section we shall show  that the more general \cds FIXED POINT decision problem is a linear time problem, implying the previous findings. Here is the definition of the main concept related to decision problem {\bf D.1}: The \emph{strategic pile of a permutation}.

\begin{definition}
Consider a permutation $\pi \in \s_n$.  
If $0$ and $n$ appear in the same cycle of $C_{\pi}$, say as $(0 \ \cdots \ n \ b_1 \ b_2 \ \cdots \ b_r)$, then \emph{strategic pile} of $\pi$, denoted $\textsf{SP}(\pi)$, is the set $\{b_1,b_2, \ldots, b_r\}$. If $0$ and $n$ do not appear in the same cycle of $C_{\pi}$, set $\textsf{SP}(\pi) = \emptyset$.  
\end{definition}

Some immediate consequences of this definition include:
\begin{lemma}\label{anpilelemma}
Fix a permutation $\pi = \lbrack a_1 \ a_2 \ \cdots \ a_n\rbrack$. If \, $0$ and $n$ appear in the same cycle of $C_{\pi}$, then $a_1-1$ and $a_n$ also appear in that same cycle, and $a_1-1, a_n\in \textsf{SP}(\pi)$.
\end{lemma}
\begin{proof}
Recall that $C_{\pi}(n) = Y_{\pi}X(n) = Y_{\pi}(0) = a_n \in \{1,\ldots,n\}$, and in particular $C_{\pi}(n) \neq 0$. Similarly, $C_{\pi}(a_1-1) = Y_{\pi}X(a_1-1) = Y_{\pi}(a_1) = 0$. Thus, if $0$ and $n$ are in the same cycle of $C_{\pi}$, then this cycle is of the form $(0 \ \cdots \ n \ a_n \cdots \ a_1-1)$, and thus $a_1-1,\;a_n\in\textsf{SP}(\pi)$. 
\end{proof}

Lemma \ref{anpilelemma} provides the following characterization of permutations whose strategic pile consists of a single element:
\begin{corollary}\label{oneptstrpile}
The strategic pile $\textsf{SP}(\pi)$ of $\pi = [a_1 \ \cdots \ a_n] \in \s_n$ has exactly one element if and only if $a_1-1 = a_n < n$.
\end{corollary}

\begin{proof}
Since $C_{\pi}(n) = a_n$, and since $C_{\pi}(a_1-1) = 0$, it follows that if the strategic pile of $\pi$ is non-empty, then $\{a_1-1,a_n\} \subseteq \textsf{SP}(\pi)$. Then a strategic pile of size one forces $a_1 - 1 = a_n$. On the other hand, if $a_1 - 1 = a_n$, then $C_{\pi}(n) = a_n$ and $C_{\pi}(a_n) = 0$, and so $\textsf{SP}(\pi) = \{a_n\}$.

We must have $a_n < n$, for if $a_n = n$ then $C_{\pi}(a_n) = C_{\pi}(n) = n = a_n $, whence $n$ is not in the same cycle of $C_{\pi}$ as $0$ and $\textsf{SP}(\pi) = \emptyset$. 
\end{proof}

The strategic piles of \cds fixed points other than the identity permutation are particularly simple.
\begin{corollary}\label{strpilecdsfp}
If $\varphi$ is the \cds fixed point $\lbrack k \ \cdots \ n \ 1 \ \cdots \ (k-1)\rbrack$ for $k > 1$, then $\textsf{SP}(\varphi) = \{k-1\}$. 
\end{corollary}

If the strategic pile of a permutation has more than one element, one can also tell some of the elements of the strategic pile beforehand:
\begin{corollary}\label{2elstrpile}
If the permutation $\pi = \lbrack a_1 \ a_2 \ \cdots \ a_n\rbrack$ has a strategic pile with more than one element, then both $a_1-1$ and $a_n \neq a_1-1$ are members of the strategic pile.
\end{corollary}
\begin{proof}
Apply Lemma \ref{anpilelemma} and Corollary \ref{oneptstrpile}. 
\end{proof}

\begin{repexample}{ex:computing C_{pi}}
For the permutation $\pi = [4 \ 2 \ 6 \ 7 \ 1 \ 3 \ 5]\in\s_7$, we computed 
that $C_{\pi} = (0 \ 7 \ 5 \ 2 \ 1 \ 4 \ 3)(6)$.  
Thus, $\textsf{SP}(\pi) = \{5,2,1,4,3\}$.
\end{repexample}

\begin{repexample}{ex:fixedptdecision}
For $\pi =  [4 \ 1 \ 3 \ 2]\in \s_4$ we have $C_{\pi} = (0 \ 4 \ 1 \ 3 \ 2)$, and so ${\textsf{SP}}(\pi) = \{ 1,\; 3,\; 2\}$.
\end{repexample}

The following is an important observation about the interaction between strategic piles and the \cds operation.

\begin{lemma}\label{spsubset}
Suppose that $\sigma$ is obtained from $\pi$ by an application of \cds for context $\{(x,x+1),(y,y+1)\}$. Then
$$\textsf{SP}(\pi) \setminus \{x,y\} \subseteq \textsf{SP}(\sigma) \subseteq \textsf{SP}(\pi).$$
\end{lemma}

\begin{proof}
This is a result of Lemma~\ref{lem:how cds affects cycles}.
\end{proof}

Thus, there are limitations on the effect of an application of \cds on the strategic pile:
\begin{corollary}[\cds Bounded Removal] \label{cdsbdedremoval}
Applying \cds removes at most two elements from the strategic pile of a permutation.
\end{corollary}

Thus, we have the following relationship between \cds fixed points derivable from a permutation $\pi$, and the strategic pile of $\pi$:
\begin{lemma}\label{cdsfixedptstrpile}
Consider a permutation $\pi \in \s_n$ that is not \cds-sortable. For each \cds fixed point $[k \ \cdots \ n \ 1 \ \cdots \ (k-1)]$ obtainable from $\pi$ by \cds operations, we have $k-1 \in \textsf{SP}(\pi)$.
\end{lemma}

To obtain an efficient solution to decision problem {\bf D.1} we shall show that conversely, each strategic pile element of a permutation $\pi$ represents a \cds fixed point obtainable from applications of \cds to $\pi$. This is the next step.

We divide the solution of the decision problem {\bf D.1}, \cds~{\bf FIXED POINT}, into two parts according to whether the strategic pile of the original permutation is the empty set, or not. As noted before the special case when the strategic pile is the empty set has been treated before using different mathematical structures. In the interest of a self-contained paper, the details for this case are now provided:

\begin{theorem}[\cds-Sortability]\label{swapsortable}
A permutation $\pi \in \s_n$ is \cds-sortable if and only if the strategic pile of $\pi$ is empty (\emph{i.e.}, $0$ and $n$ are in disjoint cycles of $C_{\pi}$).
\end{theorem}

\begin{proof}
Let $\varphi$ be a \cds fixed point of $\pi$. Because there are no $(0,1)$ 
pointers in an element of $\s_n$, Lemma~\ref{lem:how cds affects cycles} shows that applications of \cds to $\pi$ will not change whether $0$ and $n$ appear in the same cycle of the product $C_{*} = Y_{*}X$. In other words, $0$ and $n$ are in the same cycle of $\varphi$ if and only if they are in the same cycle of $\pi$. It follows that the strategic pile of $\pi$ is empty if, and only if, the strategic pile of $\phi$ is empty. As the strategic pile of a \cds fixed point other than the identity permutation, say $[k \ (k+1) \ \cdots n \ 1 \ \cdots \ (k-1)]$, is $\{ k-1\}$ and thus nonempty, it follows that $\varphi$ is the identity permutation, completing the proof.
\end{proof}

Next we consider the case when the strategic pile of $\pi$ is non-empty. We analyze how to deliberately affect a nonempty strategic pile.

\begin{theorem}[Strategic Pile Removal]\label{pileremovalth}
Let $\pi \in \s_n$ be a permutation with $|\textsf{SP}(\pi)| > 1$. For each pointer $p = (x,x+1)$ corresponding to a strategic pile element $x \in \textsf{SP}(\pi)$, there exists a pointer $q$ such that \cds is applicable to the pointer context $\{p,q\}$, and this application of \cds results in a permutation $\sigma$ for which $\textsf{SP}(\sigma) \subseteq \textsf{SP}(\pi) \setminus \{x\}$.
\end{theorem}

\begin{proof}
Consider such a $\pi = [a_1 \ \cdots \ a_n] \in \s_n$. By Lemma \ref{2elstrpile}, $(a_1-1), a_n \in \textsf{SP}(\pi)$. 

Now consider any $x \in \textsf{SP}(\pi)$.
Suppose first that $\pi$ has the form $[\cdots (x+1) \ \cdots \ x \ \cdots]$. Consider all pointers appearing between $x+1$ and $x$ in $\pi$. If they each appear twice in this portion of the inverse image notation for $\pi$, then this region would include $x+2$ and $x-1$, as well as $x+3$ and $x-2$, and so on. In other words, we would have $a_1 = x+1$ and $a_n = x$. However, this would force $\textsf{SP}(\pi) = \{a_n\}$, which contradicts the assumption that the strategic pile contains more than one element. Thus there is some pointer $p$ so that $p$ and $(x,x+1)$ appear in alternating order in $\pi$. If we apply \cds to $\pi$ for this context, then the resulting $\sigma$ will have the form $[\cdots \ x \ (x+1) \ \cdots]$, and the adjacency $x(x+1)$ will force $x \not\in \textsf{SP}(\sigma)$.

Now suppose that $\pi$ has the form $[\cdots \ x \ \cdots \ (x+1) \ \cdots]$. If $x$ and $x+1$ are not adjacent in $\pi$, then we can argue analogously to the previous case.

If, on the other hand, $\pi = [\cdots \ x \ (x+1) \ \cdots]$, then $C_{\pi}(x) = x$, and so in fact $x \not\in \textsf{SP}(\pi)$.
\end{proof}
  
\begin{example}\label{ex:reverseorder}
The permutation $\pi = [2n \ (2n-1) \ \cdots \ 2 \ 1] \in \s_{2n}$ has strategic pile
$\textsf{SP}(\pi) = \{1,\, 3,\, \cdots,\, (2n-1) \}$.  For each $x$ and $y$ distinct elements of $\textsf{SP}(\pi)$, the corresponding pointers $p=(x,\, x+1)$ and $q = (y,\, y+1)$ are \emph{not} in the context $\dots p \dots q \dots p \dots q$. Thus, any application of \cds removes at most one strategic pile element. If there is more than one strategic pile element, then Theorem \ref{pileremovalth} implies that for any strategic pile element $x$ there is an application of \cds that does \emph{not} remove $x$ from the strategic pile. 
\end{example}
Our next result shows that the conclusion of Example \ref{ex:reverseorder} holds more generally for any permutation with strategic pile larger than $2$.

\begin{theorem}[Strategic Pile Retention]\label{pilesurvivalth}
Let $\pi \in \s_n$ be a permutation with $|\textsf{SP}(\pi)| > 1$. For each element $x \in \textsf{SP}(\pi)$, there exist pointers $\{p,q\}$ such that \cds is applicable to $\pi$ for context $\{p,q\}$, and this application of \cds results in a permutation $\sigma$ for which $x \in \textsf{SP}(\sigma)$.
\end{theorem}

\begin{proof} 
If $|\textsf{SP}(\pi)| = 2$, then the result follows from Theorems~\ref{swapsortable} and~\ref{pileremovalth}.

Now suppose that $|\textsf{SP}(\pi)| > 2$, and that $x\in \textsf{SP}(\pi)$. We want to apply \cds to $\pi$ in a context that would result in a permutation $\sigma$ for which we have $x \in \textsf{SP}(\sigma)$. Let $p$ denote the pointer $(x,x+1)$. We must show that there is an application of \cds which does not result in the adjacency $\lbrack\cdots x\, x+1 \cdots\rbrack$.

First consider the case where $\textsf{SP}(\pi)$ has exactly three elements, say $i$, $j$ and $k$. We may assume that in $C_{\pi}$ we have in the cycle starting with $n$, the following: $(n \, i\, j\, k\, 0\, \dots)$. It follows that $\pi$ has one the two forms
\[
  \pi = \lbrack (k+1) \cdots  k\, (j+1) \cdots j\, (i+1) \cdots i\rbrack
\]
or
\[
  \pi = \lbrack (k+1) \cdots  j\, (i+1) \cdots k\, (j+1) \cdots i\rbrack
\]
In the latter case any two of the pointers $p$, $q$ and $r$ provides a context for an application of \cds that results in a permutation for which the remaining strategic pile term of $\pi$ is still a strategic pile element.

In the former case no pair of the pointers $p = (i,\, i+1)$, $q = (j,\, j+1)$ or $r=(k,\, k+1)$ provides a context for an application of \cds. By Theorem~\ref{pileremovalth} there is for each of $i$, $j$ or $k$ a pointer context such that applying \cds for this context produces a permutation which does not have the corresponding term on its strategic pile. Thus, suppose for example that we wish to have $i$ a member of the strategic pile of a permutation resulting from applying \cds to $\pi$. By Theorem~\ref{pileremovalth} we fix a pointer $s = (a,\, a+1)$ such that the pointer pair $\{q,\, s\}$ provides a context for an application of \cds to $\pi$. Thus in this case $\pi$ has one of the following four forms
\[
  \pi = \lbrack (k+1) \cdots a\, \cdots k\, (j+1) \cdots (a+1)\, \cdots j\, (i+1) \cdots i\rbrack,
\]
\[
  \pi = \lbrack (k+1) \cdots (a+1)\, \cdots k\, (j+1) \cdots a\, \cdots j\, (i+1) \cdots i\rbrack,
\]
\[
  \pi = \lbrack (k+1) \cdots  k\, (j+1) \cdots a\, \cdots  j\, (i+1) \cdots (a+1)\,\cdots i\rbrack,
\]
or
\[
  \pi = \lbrack (k+1) \cdots  k\, (j+1) \cdots (a+1)\, \cdots  j\, (i+1) \cdots a\,\cdots i\rbrack.
\]
We illustrate the outcome of \cds in the second case, leaving the remaining three to the reader. The result of applying \cds to the pointer pair $\{q,\, s\}$ in $\pi$ results in
\[
  \sigma = \lbrack (k+1)\,\cdots {\mathbf{\cdots j}}\, (j+1) \cdots a\, {\mathbf{(a+1)\, \cdots k}}\, (i+1) \cdots i\rbrack,
\]
which no longer has $j$ in its strategic pile, but still has $i$ (and $k$) in its strategic pile.

Now consider the case where $\textsf{SP}(\pi)$ has more than three elements. Consider an arbitrary element $j$ of $\textsf{SP}(\pi)$. By Theorem~\ref{pileremovalth} there is for each strategic pile element $k\neq j$ a pointer $p=(a,\, a+1)$ in correct context with $q = (k,\, k+1)$, such that applying \cds for the pair $\{p,\, q\}$ results in a permutation $\sigma$ which no longer has $k$ in $\textsf{SP}(\sigma)$. We must show that for some $k\neq j$ in $\textsf{SP}(\pi)$ a pointer $p$ for which $a\neq j$ can be found.

{\flushleft{\tt Hypothesis: }} Suppose that on the contrary for each $k\neq j$ that is a member of the strategic pile of $\pi$, the only pointer $p$ for which there is a context for applying \cds, is $p = (j,\, j+1)$. This context is one of $\dots p\, \dots q\, \dots p\, \dots q\, \dots$ or  $\dots q\, \dots p\, \dots q\, \dots p\, \dots$.

{\flushleft{\bf Subcase 1:}} Assume that $\pi$ is of the form $\lbrack \cdots (j+1) \cdots j\rbrack$. Consider two strategic pile elements $k$ and $\ell$ positioned as follows in $C_{\pi}$: $(n,\, j,\, k,\, \ell,\, \cdots)$. Then, by our hypothesis, $\pi$ is of the form
\[
  \lbrack \cdots k\, (j+1)\,\cdots\,(k+1)\cdots\, j \rbrack.
\]
Including also the strategic pile element $\ell$ we see that $\pi$ must be of the form
\[
  \lbrack \cdots k\, (j+1)\,\cdots\,\ell\,(k+1)\cdots\, j \rbrack.
\]
But then $\ell+1$ cannot be positioned before $j+1$ in $\pi$, since this would violate the hypothesis that only $j$ from the strategic pile provides a pointer in correct context with the pointer associated with $k$ for an application of \cds. Similarly, $\ell+1$ cannot be positioned to the right of $k+1$ in $\pi$. But then the pointers $(\ell,\, \ell+1)$ and $(j,\, j+1)$ are not in correct context for an application of \cds, contradicting the hypothesis. It follows that Subcase 1 cannot occur.

{\flushleft{\bf Subcase 2:}} Assume that $\pi$ is of the form $\lbrack (j+1) \cdots j  \cdots\rbrack$.

Consider two strategic pile elements $k$ and $\ell$ positioned as follows in $C_{\pi}$: $(n\,\cdots\, \ell\, k\, j\, 0\, \cdots)$. Then $\pi$ must be of the form
\[
  \lbrack (j+1)\, \cdots\, j\, (k+1)\,\cdots\rbrack.
\]
By the hypothesis on pointer contexts, it follows that $\pi$ is of the form
\[
  \lbrack (j+1)\, \cdots\, k\, \cdots\, j\, (k+1)\,\cdots\rbrack,
\]
and thus of the form
\[
  \lbrack (j+1)\, \cdots\, k\, (\ell+1)\, \cdots\, j\, (k+1)\,\cdots\rbrack.
\]
Now we consider the placement of the term $\ell$ in $\pi$. By hypothesis it cannot be left of $k$, and it cannot be right of $(k+1)$. Thus, $\ell$ must also be positioned between $j$ and $(j+1)$, contradicting the hypothesis that only the pointers $(j,\,j+1)$ and $(\ell,\, \ell+1)$ support a \cdr operation that eliminates $\ell$ from the strategic pile. It follows that Subcase 2 also cannot occur.

{\flushleft{\bf Subcase 3:}} Assume that $\pi$ is of the form $\lbrack \cdots\, (j+1) \cdots j  \cdots\rbrack$ or $\lbrack \cdots\, j \cdots (j+1)  \cdots\rbrack$, with other entries in each of the $\cdots$ -marked regions. We consider the second case, leaving the first case to the reader. Consider two additional elements $k$ and $\ell$ that occur in the strategic pile as follows $(n\, \cdots \, k\, j\, \ell\, \cdots \, 0\, \cdots)$. Then $\pi$ has the form
\[
  \lbrack \cdots j\,(k+1)\cdots (j+1)\, \cdots \rbrack
\]
and so by hypothesis $k$ is either to the left of $j$, or to the right of $(j+1)$ in $\pi$. We discuss the former possibility, leaving the latter to the reader. Now $\pi$ is of the form
\[
  \lbrack \cdots j\,(k+1)\cdots \ell\, (j+1)\, \cdots \,k\,\cdots \rbrack
\]
By our hypothesis, $\ell+1$ must be located between $(j+1)$ and $k$. This in turn implies that there is an element $m$ of the strategic pile so that in $\textsf{C}_{\pi}$ we have $(n\, \cdots \, k\, j\, \ell\, m\, \cdots \, 0\, \cdots)$, whence $\pi$ has the form
\[
  \lbrack \cdots j\,(k+1)\cdots\, \ell\, (j+1)\, \cdots m\, (\ell+1)\,\cdots \, k\, \cdots \rbrack.
\]
Consider the position of $m+1$ in $\pi$: It cannot be to the left of $\ell$, or to the right of $\ell+1$, as then the pointer pair $\{(\ell,\,\ell+1),\, (m,\,m+1)\}$ would provide a context for applying \cds to eliminate $\ell$ from the strategic pile, contradicting the hypothesis that $(j,\, j+1)$ is the only such pointer. But then for the strategic pile element $m$ the pointer pair $\{(j,\, j+1),\,(m,\,m+1)\}$ does not provide a context for applying \cds to eliminate $m$ from the strategic pile of $\pi$, contradicting our hypothesis.

The hypothesis leads in each case to a contradiction, whence the hypothesis is false. This completes the proof of the theorem.
\end{proof}

The Strategic Pile Removal Theorem and the Strategic Pile Retention Theorem imply the following linear time solution to the \cds FIXED POINT decision problem.

\begin{theorem}[\cds Fixed Point]\label{strategicpilecorollary}
Fix integers $n > k > 1$ and $\pi \in \s_n$. The permutation $\varphi = [k \ \cdots \ n \ 1 \ \cdots \ (k-1)]$ is a \cds fixed point for $\pi$ if and only if $k-1 \in \textsf{SP}(\pi)$.
\end{theorem}

\begin{example}\label{ex:FPDecision} Here is a specific instance of the Fixed Point Decision Problem: Given are permutations $\varphi = \lbrack 3\, 4\, \, 5\, 6\, 7\, 1\, 2\lbrack$ and $\pi = \lbrack 4\, 6\, 2\, 7\, 1\, 3 \, 5\rbrack.$ Is $\varphi$ a \cds fixed point of $\pi$? By Theorem~\ref{strategicpilecorollary}, $\varphi$ is a \cds fixed point of $\pi$ if, and only if, $2\in\textsf{SP}(\pi)$. Calculation of $\textsf{C}_{\pi}$ shows that $\textsf{SP}(\pi) = \{3,\, 4,\, 5\}$. Since $2$ is not an element of $\textsf{SP}(\pi)$, it follows that $\varphi$ is not a \cds fixed point of $\pi$.
\end{example}

\begin{repexample}{ex:computing C_{pi}}
For $\pi = [4 \ 2 \ 6 \ 7 \ 1 \ 3 \ 5]\in\s_7$ we showed earlier that $\textsf{SP}(\pi) = \{5,2,1,4,3\}$, and so the \cds fixed points of $\pi$ are the permutations
$\lbrack 6 \ 7 \ 1 \ 2 \ 3 \ 4 \ 5\rbrack$,  $\lbrack 5 \ 6 \ 7 \ 1 \ 2 \ 3 \ 4 \rbrack$, $\lbrack 4 \ 5 \ 6 \ 7 \ 1 \ 2 \ 3 \rbrack$, $\lbrack 3 \ 4 \ 5 \ 6 \ 7 \ 1 \ 2 \rbrack$, and
$\lbrack 2 \ 3 \ 4 \ 5 \ 6 \ 7 \ 1 \rbrack$,
\end{repexample}

\begin{repexample}{ex:fixedptdecision}
For $\pi =  [4 \ 1 \ 3 \ 2]\in \s_4$ we have ${\textsf{SP}}(\pi) = \{ 1,\; 3,\; 2\}$. Thus the \cds fixed points of $\pi$ are $\lbrack 4  \ 1 \ 2 \ 3 \rbrack$, $\lbrack 3 \ 4  \ 1 \ 2 \rbrack$ and $\lbrack 2 \ 3 \ 4  \ 1 \rbrack$.
\end{repexample}

\section{Invariants and structural aspects of \cds}\label{sec:cds implications} 

Although the \cds Inevitability Theorem is formulated in terms of \cds-sortability, the fundamental result is that the emptiness or not of the strategic pile of a permutation is invariant under applications of \cds.  
\begin{theorem}[\cds Inevitability Theorem]\label{cdsinevitable}
If $\pi\in \s_n$ is a \cds-sortable permutation, then regardless of the order in which \cds operations are applied, the final fixed point reached is the identity permutation of $\s_n$.
\end{theorem}
\begin{proof}
Apply Corollary \ref{cdsuniversal}, Lemma \ref{spsubset} and Theorem \ref{swapsortable}.
\end{proof}

Closely related to the \cds Inevitability Theorem is what we call the \cds Duration Theorem below:  Recall the permutation $C_{\pi} = Y_{\pi}X$ of Definition~\ref{defn:permutation product}, and set 
$$c(\pi) = \text{\ number of cycles  in the disjoint cycle decomposition of } C_{\pi}.$$

The result of \cite[Theorem 4]{DC}  can be stated in our context  in terms of $c(\pi)$ as follows:

\begin{theorem}[\cds Duration {\cite[Theorem 4]{DC}}]\label{cdssteps}
For each $\pi \in \s_n$ that is not a \cds fixed point, the 
number of applications of \cds resulting in a fixed point is
$$\begin{cases}
\frac{n+1-c(\pi)}{2} & \text{if } \pi \text{ is \cds-sortable, and}\\
\frac{n+1-c(\pi)}{2} - 1 & \text{otherwise.}
\end{cases}$$
\end{theorem}

In \cite{DC} the \cds Duration Theorem was stated as a minimality result as sorting there permitted block interchanges different from \cds.

Another sortability invariant is obtained by considering the relationship between $C_{\pi}$ and $C_{\pi^{-1}}$. Using the fact that the cycle notation for a permutation $\sigma^{-1}$ can be obtained from that of $\sigma$ by reversing the data in each cycle, the following result follows from 
the definition of $C_{\pi}$.

\begin{lemma}\label{findingcpiinv}
For a permutation $\pi \in \s_n$, the permutation $C_{\pi^{-1}}$ is obtained from $C_{\pi}^{-1}$, the inverse of $C_{\pi}$, by replacing each $i \in \{1,\ldots,n\}$ with $\pi(i)$.
\end{lemma}

\begin{repexample}{ex:computing C_{pi}}
Consider $\pi = [4 \ 2 \ 6 \ 7 \ 1 \ 3 \ 5] \in \s_7$. From this inverse image presentation of $\pi$ we find that $\pi(1) = 5$, $\pi(2) = 2$, $\pi(3) = 6$, $\pi(4) = 1$, $\pi(5) = 7$, $\pi(6) = 3$ and $\pi(7) = 4$. Thus, $\pi^{-1} = [5 \ 2 \ 6 \ 1 \ 7 \ 3 \ 4]$.
Earlier we computed $C_{\pi} = (0 \ 7 \ 5 \ 2 \ 1 \ 4 \ 3)(6)$. Thus $C_{\pi}^{-1} = (0 \ 3 \ 4 \ 1 \ 2 \ 5 \ 7)(6)$, and $C_{\pi^{-1}} = (0 \ 6 \ 1 \ 5 \ 2 \ 7 \ 4)(3)$.
\end{repexample}

From the characterization in Lemma \ref{findingcpiinv} we obtain the fact that \cds-sortability of a permutation $\pi$ is invariant with respect to the group $\s_n$'s inverse operation. 

\begin{corollary}\label{inversecdssortable}
A permutation $\pi$ is \cds-sortable if and only if $\pi^{-1}$ is \cds-sortable. 
\end{corollary}

\begin{proof}
Because $(\pi^{-1})^{-1} = \pi$, we need only prove one direction of the biconditional statement.

Suppose that $\pi = [a_1 \ \cdots \ a_n] \in \s_n$ is not \cds-sortable. Thus $0$ and $n$ appear in the same cycle of $C_{\pi}$ by Theorem~\ref{swapsortable}. Recall that $a_n = \pi^{-1}(n)$ must be in this cycle as well. These values $\{0,n,a_n\}$ necessarily appear in the same cycle of $C_{\pi}^{-1}$, although in a different cycle order. To obtain $C_{\pi^{-1}}$ from $C_{\pi}^{-1}$, we replace each $i \in \{1,\ldots,n\}$ by $\pi(i)$, and thus $\{0,\pi(n),\pi(a_n) = n\}$ are in the same cycle of $C_{\pi^{-1}}$. Therefore $\pi^{-1}$ is not \cds-sortable, completing the proof.
\end{proof}

\begin{repexample}{ex:computing C_{pi}} 
Although the non-emptiness the strategic pile of a permutation is an invariant of the inverse operation of the group $\s_n$, the cardinality of the strategic pile is not an invariant. For example, the strategic pile of $\pi = [4 \ 2 \ 6 \ 7 \ 1 \ 3 \ 5] \in \s_7$ is $\{1,\; 2,\; 3,\; 4,\; 5\}$, while the strategic pile of $\pi^{-1} = [5 \ 2 \ 6 \ 1 \ 7 \ 3 \ 4]$ is $\{ 4\}$.
\end{repexample}

Another \cds invariant emerges from considering parity features:

\begin{lemma}[\cds Parity Invariance]\label{cdsparityinvariance}
Let $\pi$ be a permutation with either of the following properties:
\begin{enumerate}
\item $\pi(j) \bmod 2=0 \iff j \bmod 2=0$ for all $j$, or
\item $\pi(j) \bmod 2=1 \iff j \bmod 2=0$ for all $j$.
\end{enumerate}
Let $\sigma$ be obtained by applying \cds to $\pi$ with some context. Then $\sigma$ satisfies the same property above.
\end{lemma}

\begin{proof}
This follows from a case analysis of the parities of $x$ and $y$ in the possibilities outlined in Table~\ref{table:cds}.
\end{proof}

To determine if a permutation $\pi \in \s_n$ is \cds-sortable, we only need construct the permutation $C_{\pi}$ and check whether $0$ and $n$ appear in the same cycle, a linear time computation. However, in practical situations there are a number of easily recognizable features of a permutation $\pi$ that makes it unnecessary to calculate a cycle of the permutation $C_{\pi}$:

\begin{corollary}\label{position1}
If $\pi = [a_1 \ a_2 \ \cdots \ a_n] \in \s_n$ satisfies $a_1 = 1$ or $a_n = n$, then $\pi$ is \cds-sortable.
\end{corollary}

\begin{proof}
If $a_1 = 1$, then the permutation $C_{\pi}$ fixes $0$. If $a_n = n$, then the permutation $C_{\pi}$ fixes $n$. In either case, $0$ and $n$ necessarily appear in separate cycles of the disjoint cycle decomposition of $C_{\pi}$.
\end{proof}

\begin{corollary} \label{parityswitch}\
\begin{enumerate}\renewcommand{\labelenumi}{(\alph{enumi})}
\item If $n$ is odd and $\pi \in \s_n$ satisfies
$\lbrack \pi(j) \bmod 2 = 0 \iff j \bmod 2 = 0\rbrack$,
then $\pi$ is \cds-sortable.
\item If $\pi \in \s_n$ satisfies
$\lbrack \pi(j) \bmod 2 = 1 \iff j \bmod 2 = 0\rbrack$,
then $n$ is even and $\pi$ is not \cds-sortable.
\end{enumerate}
\end{corollary}

\begin{proof} \
\begin{enumerate}\renewcommand{\labelenumi}{(\alph{enumi})}
\item The cycle containing $0$ in the disjoint cycle decomposition of $C_{\pi}$ will contain only even values. If $n$ is odd, then it will not appear in this cycle, and thus $\pi$ is \cds-sortable by Theorem~\ref{swapsortable}.
\item The entry $1$ appears in an even position of $\pi$. By Lemma~\ref{cdsparityinvariance} repeated applications of \cds will lead to a fixed point $\varphi$ with $1$ in an even numbered position, and so $\varphi$ is not the identity and so $\pi$ is not \cds-sortable. The parity of $n$ follows from Lemma~\ref{cdsparityinvariance} and the fact that $1$ and $n$ are adjacent in $\varphi$.
\end{enumerate}
\end{proof}

\begin{example}
The permutation $\pi = \lbrack 1 \ 4 \ 7 \ 2 \ 5 \ 8 \ 3 \ 6\rbrack$ is an element of $\s_8$ and satisfies the biconditional statement of Corollary~\ref{parityswitch}(a), and is \cds-sortable by Corollary~\ref{position1}.
The permutation $\mu = \lbrack 5 \ 6 \ 7 \ 2 \ 1 \ 8 \ 3 \ 4\rbrack$ is an element of $\s_8$ and satisfies the biconditional statement of Corollary~\ref{parityswitch}(a), and is not \cds-sortable by Corollary~\ref{oneptstrpile}.  Thus the hypothesis that $n$ is odd in Corollary~\ref{parityswitch} is necessary, but not sufficient. Note however that $\mu$ has several adjacencies, and from the point of view of \cds sorting is essentially the permutation $\lbrack 4 \ 2 \ 1 \ 5 \ 3 \rbrack$ for which $n = 5$ is odd. 
\end{example}

\begin{repexample}{ex:reverseorder}
The permutation $[n \ (n-1) \ \cdots \ 2 \ 1] \in \s_n$ is \cds-sortable if and only if $n$ is odd. For if $n$ is even, then Lemma \ref{cdsparityinvariance} implies that in any \cds fixed point reached, $1$ is in an even position, and thus the fixed point is not the identity. Thus, \cds-sortability implies $n$ is odd. Conversely, if $n$ is odd then Corollary \ref{parityswitch} implies that this permutation is \cds-sortable.
\end{repexample}

Our next example, Example \ref{ex:sortabilityandcomposition}, illustrates that although the \cds-sortable subset of the group $\s_n$ is closed under taking inverses, it is not closed under the group operation of $\s_n$, and so does not form a subgroup of $\s_n$.
\begin{example}\label{ex:sortabilityandcomposition}
Consider the permutations $\pi = \lbrack 1 \ 4 \ 2 \ 5 \ 3\rbrack$ and $\sigma = \lbrack 3 \ 2 \ 4 \ 1 \ 5\rbrack$. Each is \cds-sortable. The composition $\sigma\circ\pi = \lbrack 2 \ 4 \ 5 \ 1 \ 3 \rbrack$ has strategic pile $\{ 1,\ 2,\ 3\}$ which is not empty, and thus $\sigma\circ\pi$ is not \cds-sortable.
\end{example}

\section{\cds and games}\label{section:cdsgames}

When an unsigned permutation $\pi$ is not \cds-sortable, the resulting fixed point reached after successive \cds operations may depend on the order in which these sorting operations are applied. This phenomenon suggests several combinatorial sorting games.

Towards defining one class of such games, fix a permutation $\pi$ and a set \textsf{F} of \cds fixed points of $\pi$. Then the two-person game
$\textsf{CDS}(\pi,\textsf{F})$
is played as follows.

\begin{quote}
Player ONE applies a \cds operation to $\pi$ to produce $\pi_1$. Player TWO applies a \cds operation to $\pi_1$ to produce $\pi_2$, and so on. The players alternate \cds moves in this manner until a fixed point $\varphi$ is reached. Player ONE wins if $\varphi \in \textsf{F}$. Otherwise, player TWO wins.
\end{quote}

Each application of a \cds move to a permutation that is not a \cds fixed point reduces the number of non-adjacencies. Elements of $\s_n$ have at most $n-1$ non-adjacencies, and thus there are always at most $n-1$ moves in the game for $\pi \in \s_n$.
Therefore, for each $n$, these games are of finite length, and none of them ends in a draw. A classical theorem of Zermelo \cite{Z} implies that for each choice of $\pi$ and \textsf{F}, some player has a winning strategy in the game $\textsf{CDS}(\pi,\textsf{F})$:

\begin{theorem}[Zermelo]\label{fundthfingames}
For any finite win-lose game of perfect information between two players, one of the players has a winning strategy.
\end{theorem}

This fact suggests the following decision problem:
\begin{quote}
\framebox{
\parbox{.8\textwidth}{
{\bf D.2 \hspace{0.1in} FIXED POINT SORTING GAME:}

INSTANCE: A positive integer $n$, a permutation $\pi$, \and a set \textsf{F} of \cds-fixed points.

QUESTION: Does ONE have a winning strategy in the game $\textsf{CDS}(\pi,\textsf{F})$?}}
\end{quote}

The complexity of the Fixed Point Sorting Game decision problem is currently unknown and appears to depend strongly on the structure of the strategic pile, as well as the structure of the set \textsf{F} of \cds fixed points that are favorable to player ONE. When the strategic pile has at most $2$ elements, the Fixed Point Sorting Game decision problem can be dealt with by reformulating the \cds-Sortability result in terms of a fixed point sorting game.

\begin{repptheorem}{swapsortable}\label{swapsortablereformulated}
Let $n$ be a positive integer. Let $\textsf{F} = \{[1 \ 2 \ \cdots \ n]\}$. For any $\pi \in \s_n$, the following statements are equivalent.
\begin{enumerate}\renewcommand{\labelenumi}{(\alph{enumi})}
\item ONE has a winning strategy in the game $\textsf{CDS}(\pi,\textsf{F})$.
\item No cycle in $C_{\pi}$ contains both $0$ and $n$.
\item The strategic pile of $\pi$ is empty.
\end{enumerate}
\end{repptheorem}

Since criterion (b) can be verified in time linear in $n$, deciding whether ONE has a winning strategy in this game is of linear time complexity.

More generally we have the following proposition:
\begin{proposition}\label{smallpile}
Let $n$ as well as $\pi\in \s_n$ be such that $\vert\textsf{SP}(\pi)\vert\le 2$. For each nonempty set $\textsf{F}$ of \cds fixed points of $\pi$, ONE has a winning strategy in the game $\textsf{CDS}(\pi,\textsf{F})$.
\end{proposition}
\begin{proof}
When the strategic pile $\textsf{SP}(\pi)$ is nonempty and has at most two elements, ONE has a winning strategy if and only if $\textsf{F}$ is nonempty, by Theorem~\ref{pilesurvivalth}. 
\end{proof}
When the strategic pile $\textsf{SP}(\pi)$ has more than two elements, additional factors influence who has a winning strategy. We illustrate this feature in the following example.

\begin{repexample}{ex:reverseorder} 
Consider the permutation
$$\alpha_n = [2n \ (2n-1) \ \cdots \ 3 \ 2 \ 1].$$
Its strategic pile is $\{1,3,5,\ldots,2n-1\}$, and no pair of pointers from $\{ (1,2), (3,4), \ldots, (2n-1, 2n)\}$ provides context to a \cds-application. Thus a player can remove at most one element from the strategic pile on any given turn. 
{\flushleft{\bf Claim:}} ONE has a winning strategy in $\textsf{CDS}(\alpha_n,\textsf{F})$ if, and only if, $|\textsf{F}| \ge n/2$. \\
Indeed, ONE's strategy is always to remove, whenever one exists, a strategic pile element that is not also in $\textsf{F}$. That this can be done follows from the Strategic Pile Removal Theorem, Theorem~\ref{pileremovalth}. Also, player TWO can remove at most one element of $\textsf{F}$ on each turn. Since ONE takes the first turn, and the strategic pile contains $n$ elements, this is a winning strategy if $|\textsf{F}| \ge n/2$. For the converse, suppose that $\vert\textsf{F}\vert < n/2$. After ONE's first move (removing an element of the strategic pile not belonging to \textsf{F}), player TWO has the role of player ONE in the game based on the resulting permutation. At least half of the elements of the corresponding strategic pile belong to this new player ONE, and we can now apply the previous argument.

For the promised contrast to the case when the strategic pile has at most two elements, now consider the specific instance of this example when $n=3$. The permutation $\alpha_3 = \lbrack 6\; 5\; 4\; 3\; 2\; 1 \rbrack$ has strategic pile $\{1,\; 3,\; 5\}$, and if the subset $F$ of the strategic pile has only one element, then TWO has a winning strategy in the game $\textsf{CDS}(\alpha_3,\textsf{F})$.
\end{repexample}
In Example \ref{ex:reverseorder} we exploited the fact that at most one strategic pile element is removed in a \cds move. Not all strategic piles have this property, but by Corollary \ref{cdsbdedremoval} at most two elements of the strategic pile can be removed per application of a \cds sorting operation. This fact leads to the criterion in Theorem~\ref{cdsgreedy} below.

\begin{theorem}\label{cdsgreedy}
Let $\pi \in \s_n$ not be \cds-sortable, and consider some $\textsf{F} \subseteq \textsf{SP}(\pi)$.
\begin{enumerate}\renewcommand{\labelenumi}{(\alph{enumi})}
\item{ If $|\textsf{F}| \ge \frac{3}{4}|\textsf{SP}(\pi)|$ 
then ONE has a winning strategy in $\textsf{CDS}(\pi,\textsf{F})$.}
\item{ If $|\textsf{F}| < \frac{1}{4}|\textsf{SP}(\pi)| - 2$ 
then TWO has a winning strategy in $\textsf{CDS}(\pi,\textsf{F})$.}
\end{enumerate}
\end{theorem}

\begin{proof}
By Theorem~\ref{pileremovalth}, it is always possible to remove an element favoring the opponent from $\textsf{SP}(\pi)$, as long as $|\textsf{SP}(\pi)| > 1$ and elements favoring the opponent remain in play. The number of disjoint cycles in the permutation $C_*$ increases by at most two in a given turn, meaning that at most two elements at a time are removed from the strategic pile.

Let $m$ be the number of elements on the strategic pile. Let $m_1(k)$ denote the number of elements on the strategic pile that favor ONE before ONE's $k$th turn, and let $m_2(k)$ denote the number that favor TWO before ONE's $k$th turn. Thus $m = m_1(1) + m_2(1)$.

Suppose that $m_1(1) \ge 3m_2(1)$.

Consider the following strategy for player ONE: on each turn, remove as many elements favoring TWO as possible from the strategic pile. There are now three possibilities when ONE is about to play his $k$th turn:
\begin{enumerate}
\item $m_2(k) = 0$,
\item only one element that favors TWO is removable, or
\item two elements favoring TWO are removable.
\end{enumerate}

In the first scenario, ONE wins the resulting play.

Consider the second scenario. ONE's play is to remove that lone element favorable to TWO, even if it requires also removing an element favorable to ONE. In TWO's subsequent move, it may be possible to remove two elements favoring ONE, but as long as $m_1(k) \ge 3m_2(k)$, then $m_1(k+1) \ge 3m_2(k+1)$.

In the third scenario, ONE's play is to remove two elements that are favorable to TWO. Then, by the Strategic Pile Removal Theorem, Theorem~\ref{pileremovalth}, TWO can remove at least one element favorable to ONE on the next turn. Thus
\begin{eqnarray*}
m_1(k+1) &\ge& m_1(k) - 2\\
&\ge& 3m_2(k) - 2\\
&>& 3m_2(k+1).
\end{eqnarray*}

Thus we have $m_1(k) \ge 3m_2(k)$ for all $k$. Since $m_i(k) < m_i(k-1)$ for all $k$ during which the game is in play, the game will eventually terminate, when a $k$ is reached for which $m_2(k) = 0$.

Suppose, instead, that $m_1(1) < (1/3)m_1(2)- 2$. On the first turn, player ONE can remove at most two elements favoring player TWO. Thus $m_2(2) \ge 3m_1(2)$, and so the above argument shows that player TWO has a winning strategy in the game.
\end{proof}

At first glance it may seem that the criterion of Theorem \ref{cdsgreedy} can be vastly improved. However, it has since been proved in \cite{JSST} that this result is optimal: For each $n>4$ that is a multiple of 4, there are examples of permutations $\pi_n$ such that, in the notation of of Theorem \ref{cdsgreedy}, if $\vert\textsf{F}\vert<\frac{3}{4}\vert\textsf{SP}(\pi_n)\vert$, then TWO has a winning strategy in the game $\textsf{CDS}(\pi_n,\textsf{F})$. No other general criteria towards answering the Fixed Point Sorting Game decision problem, {\bf D.2}, are currently known.

Also the popular \emph{normal play} and \emph{misere} versions of combinatorial games can be considered for the game $\textsf{CDS}(\pi,\textsf{F})$:

\begin{definition}\
\begin{enumerate}\renewcommand{\labelenumi}{(\alph{enumi})}
\item The game $\textsf{CDSN}(\pi)$ proceeds as $\textsf{CDS}(\pi,\textsf{F})$, but the player that makes the last legal move wins. This is the \emph{normal play} version.
\item The game $\textsf{CDSM}(\pi)$ proceeds as $\textsf{CDS}(\pi,\textsf{F})$, but the player that makes the last legal move loses. This is the \emph{misere} version.
\end{enumerate}
\end{definition}

As noted before, these games are of finite length, and do not end in draws between the two players. Thus Zermelo's Theorem, Theorem~\ref{fundthfingames}, applies. The accompanying decision problems are:

\begin{quote}
\framebox{
\parbox{.8\textwidth}{
{\bf D.3 \hspace{0.1in} NORMAL SORTING GAME:}

INSTANCE: A positive integer $n$ and a permutation $\pi \in \s_n$.

QUESTION: Does ONE have a winning strategy in the game $\textsf{CDSN}(\pi)$?}}
\end{quote}

\begin{quote}
\framebox{
\parbox{.8\textwidth}{
{\bf D.4 \hspace{0.1in} MISERE SORTING GAME:}

INSTANCE: A positive integer $n$ and a permutation $\pi \in \s_n$.

QUESTION: Does ONE have a winning strategy in the game $\textsf{CDSM}(\pi)$?}}
\end{quote}

These two decision problems are solved by an application of earlier duration results related to \cds: Recall that the \cds Duration Theorem, Theorem~\ref{cdssteps}, gave the number of steps required to obtain a \cds fixed point of $\pi$. The parity of this number of steps determines which player has a winning strategy in the normal or misere version of the game. Since Christie's method for computing $c(\pi)$ is a linear time algorithm (see \cite{DC}, Figure 7) the complexity of the decision problem of which player has a winning strategy in ${\sf CDSN}(\pi)$ or in ${\sf CDSM}(\pi)$ for $\pi \in \s_n$ is linear in $n$.

\section{Sorting by context directed reversals}\label{section:cdr} 

In earlier literature the sorting operation \cdr appeared under different names: In studies of the reversal distance between signed permutations the name \emph{oriented reversals} was used, for example in \cite{AB, HP}. Independently, in literature on models of ciliate micronuclear decryption, such as for example \cite{EHPPR, EHPR}, the name \textsf{hi} was used. The \emph{context directed reversal} sorting operation, \cdr, is defined as follows in terms of the pointers described in Section~\ref{section:definitions}.

\begin{definition}
Consider a signed permutation $\pi$ in which the pointers $p$ and $-p$ both appear. The \cdr operation on $\pi$ with context $p$ reverses and negates the block of letters that are flanked by the pointers $\{\pm p\}$.
\end{definition}

\begin{example}
In $\pi = [3 \ -\!1 \ -\!2 \ 5 \ 4] \in \s_5^{\pm}$ the left pointer of $3$ is $(2,3)$, and the left pointer of $-2$ is $-(3,2)$. Applying \cdr with this context results in the signed permutation is $[1 \ -\!3 \ -\!2 \ 5 \ 4] \in \s_5^{\pm}$.
\end{example}

\begin{definition}
An element of $\s_n^{\pm}$ is \emph{\cdr-sortable} if application of some sequence of \cdr operations terminates in the identity element $e \in \s_n^{\pm}$. An element of $\s_n^{\pm}$ is \emph{reverse \cdr-sortable} if application of some sequence of \cdr operations terminates in $[-\!n \ -\!(n-1) \ \cdots \ -\!1]$.
\end{definition}

\begin{definition}
Let $\fixcdr$ be the elements of $\s_n^{\pm}$ for which \cdr cannot be applied. 
\end{definition}
Thus, $\fixcdr \subseteq \s_n^{\pm}$ are \cdr fixed points. We now identify the structure of the set \fixcdr.

\begin{theorem}\label{cdrfixedinsigned}
The set $\fixcdr \subseteq \s_n^{\pm}$ consists of elements whose inverse image notation is either entirely positive or entirely negative. Moreover, $\fixcdr$ is a subgroup of $\s_n^{\pm}$.
\end{theorem}

\begin{proof}
First note that $\fixcdr$ consists of those permutations $[a_1 \ \cdots \ a_n]$ in which $\pm p$ do not both appear for any pointer $p$. Certainly if all $\{a_i\}$ have the same sign then this will be the case. We must now consider whether any other permutation might also be an element of $\fixcdr$. Suppose that at least one of the $\{a_i\}$ is positive and at least one is negative. In fact, we can find an $x$ such that some $a_i = x$ and some $a_j = -(x+1)$. Without loss of generality, suppose that $x > 0$. Then the right pointer of $a_i$ is $(x,x+1)$, and the right pointer of $a_j$ is $-(x+1,x)$. Thus we would be able to apply \cdr with context $\pm(x,x+1)$, and such a permutation would not be a \cdr fixed point.

That $\fixcdr$ is a subgroup of $\s_n^{\pm}$ follows from the fact that the product of two same-signed integers is always positive, and the product of two opposite-signed integers is always negative.
\end{proof}

\begin{corollary}
$|\fixcdr| = 2n!$.
\end{corollary}

The set $\fixcds \subseteq \s_n^{\pm}$ consisting of \cds fixed points also has an interesting algebraic structure:
\begin{theorem}\label{cdsfixedinsignedperm}
The set $\fixcds \subseteq \s_n^{\pm}$ of \cds fixed points is isomorphic to the dihedral group $D_n$.
\end{theorem}

\begin{proof}
By analyzing possible pointer contexts, we see that the only \cds fixed points of $\s_n^{\pm}$ are of either the form
$$[k \ (k+1) \ \cdots \ n \ 1 \ 2 \ \cdots \ (k-1)],$$
as described in Lemma~\ref{lem:cdsfixedpoints}, or
$$[-\!(k-1) \ \cdots \ -\!2 \ -\!1 \ -\!n \ -\!(n-1) \ \cdots \ -\!k].$$
The collection of these objects is generated by the signed permutations
$$\mu_n = [n \ (n-1) \ \cdots \ 2 \ 1]$$
and
$$\nu_n = [-\!n \ -\!(n-1) \ \cdots \ -\!2 \ -\!1].$$
These signed permutations have order $n$ and $2$, respectively, in $\s_n^{\pm}$, and $\mu_n^i\nu_n = \nu_n\mu_n^{n-i}$. Thus $\langle \mu_n,\nu_n \rangle$ is isomorphic to the dihedral group $D_n$.
\end{proof}

As with \cds, the \cdr-sortability, or reverse \cdr-sortability, of a signed permutation $\pi$ has been treated in several prior works including \cite{HP} and independently \cite{BHR, EHPR}. But in stark contrast with  Theorem \ref{cdsinevitable} for \cds, there is no corresponding inevitability theorem for the \cdr sorting operation, as illustrated in Example \ref{ex:cdrlengths}:
\begin{example}\label{ex:cdrlengths}
Consider $\pi = [3 \ -\!1 \ -\!2 \ 5 \ 4] \in \s_5^{\pm}$. This is a \cdr-sortable element, as can be seen by applying \cdr first to the pointer $(2,3)$, then to $(3,4)$, then to $(4,5)$, and finally to $(1,2)$. However, first applying \cdr for the pointer $(2,3)$ and then for the pointer $(1,2)$ produces $[1 \ 2 \ 3 \ 5 \ 4]$, which is a \cdr fixed point and not \cdr-sortable.
\end{example}

This phenomenon makes the \cdr sorting operation more complex to analyze than the \cds sorting operation. Absence of an inevitability theorem for \cdr-sortable signed permutations necessitates a strategic approach to sorting signed permutations by \cdr. An efficient strategy for determining a sequence of applications of \cdr that successfully sort a \cdr-sortable, or a reverse \cdr-sortable, signed permutation has been identified in \cite{HP}. For an interesting exposition of, as well as further advances in the search for optimally efficient strategies towards sorting a \cdr-sortable signed permutation we recommend \cite{AB}.  

A second complication in the analysis of the \cdr sorting operation is that at this time there is no known efficient algorithm for the \cdr Fixed Point Decision Problem:
 \begin{quote}
\framebox{
\parbox{.8\textwidth}{
{\bf D.5 \hspace{0.1in} \cdr FIXED POINT:}

INSTANCE: Integer $n > 1$, a signed permutation $\pi \in \s_n^{\pm}$, and a \cdr fixed point $\varphi \in \s_n^{\pm}$.

QUESTION: Is $\varphi$ a \cdr fixed point of $\pi$?}}
\end{quote}

Though \cdr-sortability, or reverse sortability, of a signed permutation $\pi$ has been treated in \cite{HP} and independently in \cite{BHR, EHPR}, we derive here, in the interest of being self-contained, only a necessary condition for \cdr-sortability (and reverse sortability). This condition will be a tool in identifying some of the properties of \cdr fixed points of certain signed permutations. 

Towards this objective we construct an object suggested by a hybrid of the cycle graph of a permutation as introduced by Bafna and Pevzner \cite{BP}, and the breakpoint graph of a signed permutation, introduced by Hannenhalli and Pevzner \cite{HP}. Though not exactly the object constructed by \cite{DL}, our construction is in the same spirit:

For an integer $m$, define $f(m)$ by:
$$f(m) = \begin{cases}
[2m\!-\!1 \ 2m] & \text{if } m \ge 0 \text{ and}\\
[-\!2m \ -\!(2m+1)] & \text{if } m < 0\\
\end{cases}$$
For $\pi = [a_1 \ \cdots \ a_n] \in \s_n^{\pm}$, set 
$$\pi^*= [f(a_1)\| f(a_2) \| \cdots \| f(a_n)],$$
the concatenation of the ordered lists $f(a_i)$.

In analogy to Definition~\ref{defn:permutation product}, we now make the following definition.

\begin{definition}\label{defn:cdr permutation product}
Given $\pi \in \s_n^{\pm}$ with $\pi^* = [b_1 \ \cdots \ b_{2n}]$, define two permutations on the set $\{0, 1, \ldots, 2n+1\}$, written in cycle notation, as follows:
$$U = (0 \ 1)(2 \ 3)(4 \ 5) \cdots (2n \ 2n+1)$$
and
$$V_{\pi} = (0 \ b_1) (b_2 \ b_3)(b_4 \ b_5) \cdots (b_{2n} \ 2n+1).$$
Now set
$$D_{\pi} = V_{\pi}U.$$
\end{definition}

\begin{example}\label{ex:breakpoint}
If $\pi = [3 \ -\!1 \ -\!2 \ 5 \ 4]$, then
$$\pi^* = [ 5 \ 6 \ 2 \ 1 \ 4 \ 3 \ 9 \ 10 \ 7 \ 8]$$
and
$$V_{\pi} = (0 \ 5) (6 \ 2)(1 \ 4)(3 \ 9)(10 \ 7)(8 \ 11),$$
which produces
\begin{eqnarray*}
D_{\pi} &=& V_{\pi}U\\
&=& (0 \ 4)(1 \ 5)(2 \ 9 \ 11 \ 7)(3 \ 6 \ 10 \ 8).
\end{eqnarray*}

Applying \cdr with pointer $\pm(2,3)$ to $\pi$ produces $\sigma$, where
$$\sigma^* = [1 \ 2 \ 6 \ 5 \ 4 \ 3 \ 9 \ 10 \ 7 \ 8],$$
and we now have $D_{\sigma} = (0)(1)(2 \ 9 \ 11 \ 7)(3 \ 6 \ 10 \ 8)(4)(5)$.
\end{example}

Consider the permutation $D_{\pi}$ where $\pi^*$ is $\lbrack b_1\ \cdots \ b_{2n}\rbrack\in\s_{2n}$. As $D_{\pi}(2n) = b_{2n} \neq 0$, the two-cycle $(0\ 2n)$ never occurs in the disjoint cycle decomposition of $D_{\pi}$. Thus, when $0$ and $2n$ occur in the same cycle of $D_{\pi}$, then that cycle is of length at least three.  Also observe that if the entries $a$ and $a+1$ of $\pi$ appear in $\pi$ as $\lbrack \cdots a \ a+1 \ \cdots\rbrack$ then the two-cycle $(2a\ 2a+1)$ appears in (the disjoint cycle decomposition of) $V_{\pi}$, and as this two-cycle also appears in $U$, the disjoint cycle decomposition of $D_{\pi}$ contains the two one-cycles $(2a)$ and $(2a+1)$. The objective of sorting by \cdr is to convert $D_{\pi}$ to a permutation which in disjoint cycle decomposition form is a composition of one-cycles only. Since an application of \cdr creates an adjacency, the effect on the cycles of the permutation is to extract singleton cycles from existing longer cycles.

\begin{theorem}[\cdr-Sortability]\label{cdrsortability}
If the signed permutation $\pi \in \s_n^{\pm}$ is \cdr-sortable, then $0$ and $2n$ are in disjoint cycles of $D_{\pi}$.
\end{theorem}
 
\begin{proof}
Suppose that $0$ and $2n$ are in the same cycle of $D_{\pi}$. Applying \cdr to $\pi$ for any pointers $\pm(x,x+1)$ to obtain some $\sigma$ will create two $1$-cycles $(2x)(2x+1)$ in $D_{\sigma}$. The cycle $C$ of $D_{\pi}$ that contained $x$ appears in $D_{\sigma}$ as $C \setminus \{x\}$, and similarly for the cycle that had contained $x+1$, and all other cycles of $D_{\pi}$ are unchanged in $D_{\sigma}$. There is no $(0,1)$ or $(n,n+1)$ pointer in $\pi \in \s_n^{\pm}$, so if $0$ and $2n$ are in the same cycle of $D_{\pi}$ then they must be for $D_{\sigma}$ as well. Moreover, since $(2n \ 0)$ is not a cycle in either $U$ or $V_{\pi}$, there must be some third value in this cycle as well. It will never be possible to reduce the length of the cycle containing $0$ and $2n$ to less than three, and hence $\pi$ is not \cdr-sortable. \end{proof}

The converse to Theorem~\ref{cdrsortability} is not true, as illustrated in Example \ref{ex:cdrsortcharfails}:

\begin{example}\label{ex:cdrsortcharfails}
Let $\pi = [2 \ 4 \ 3 \ 5 \ -\!1 \ 6] \in \s_6^{\pm}$. We have $\pi^* = [3 \ 4 \ 7 \ 8 \ 5 \ 6 \ 9 \ 10 \ 2 \ 1 \ 11 \ 12]$ and $D_{\pi} = (0 \ 11 \  2)(1 \ 3 \ 10)(4 \ 8 \ 6)(5 \ 7 \ 9)(12)$, in which $0$ and $12$ appear in disjoint cycles. This $\pi$ is not \cdr-sortable because after the (only possible) first application of \cdr with context $\pm (1,2)$, the resulting signed permutation is $[ -\!5 \ -\!3 \ -\!4 \ -\!2 \ -\!1 \ 6]$, and after the (only possible) next application of \cdr with context $\pm (5,6)$, the resulting permutation is the unsorted $[1 \ 2 \ 4 \ 3 \ 5 \ 6]$.
\end{example}

The following results provide some constraints on the possibilities for \cdr fixed points of \cdr-sortable, or reverse sortable, signed permutations.

\begin{theorem}\label{cdrsortable}
If $\pi \in \s_n^{\pm}$ is \cdr-sortable, then every \cdr fixed point of $\pi$ is an element of $\s_n$.
\end{theorem}

\begin{proof}
Let $\pi \in \s_n^{\pm}$ be \cdr-sortable and assume, to the contrary, that a sequence of \cdr applications leads to a fixed point $\varphi = [a_1 \ \cdots \ a_n]$ with $a_i < 0$ for all $i$. Let $b_i = |a_i|$ for all $i$. Then
$$\varphi^* = [2b_1 \ (2b_1 - 1) \ 2b_2 \ (2b_2 - 1) \ \cdots \ 2b_n \ (2b_n - 1)].$$
By Theorem~\ref{cdrsortability}, the letters $0$ and $2n$ appear in disjoint cycles of $D_{\pi}$, and the same must be true for $D_{\varphi}$.

Knowing $\varphi^*$ allows us to compute $D_{\varphi}$. In particular, note that all but two of the $2$-cycles in $U$ and $V_{\varphi}$ contain values of opposite parities. Thus the symbol $0$ is preceded in its cycle of $D_{\varphi}$ by a string of odd values, and succeeded in its cycle by a string of even values. Thus there is some other even value $2x$ in the cycle that is preceded by a string of even values and succeeded by a string of odd values. Given the restrictions on $\varphi^*$, the only possibility for this $2x$ is $2n$, which contradicts Theorem~\ref{cdrsortability}.
\end{proof}

An analogous argument constrains the fixed points attainable for a reverse \cdr-sortable permutation.
 
\begin{theorem}\label{cdrreversesortable}
If $\pi \in \s_n^{\pm}$ is reverse \cdr-sortable, then the inverse image notation of every \cdr fixed point of $\pi$ consists entirely of negative values.
\end{theorem}

It is tempting to conjecture that in signed permutations $\pi$ with a cycle of $D_{\pi}$ containing both $2n$ and $0$, the segment between $2n$ and $0$ in a cycle has the same properties relative to \cdr as the strategic pile did for \cds. This, however, is not the case.

\begin{example}\label{ex:5 -2 7 4 -1 3 6}
The permutation $\pi = [5 \ -\!2 \ 7 \ 4 \ -\!1 \ 3 \ 6]$ is neither \cdr-sortable (Theorem~\ref{cdrsortable}) nor reverse \cdr-sortable (Theorem~\ref{cdrreversesortable}). The \cdr fixed points of $\pi$ are $[5 \ 6 \ 7 \ 1 \ 2 \ 3 \ 4]$, $[5 \ 1 \ 2 \ 3 \ 4 \ 7 \ 6]$, $[7 \ 4 \ 5 \ 1 \ 2 \ 3 \ 6]$, $[7 \ 1 \ 2 \ 3 \ 4 \ 5 \ 6]$, $[-\!1 \ -\!5 \ -\!4 \ -\!7 \ -\!6 \ -\!3 \ -\!2]$, and $[-\!1 \ -\!7 \ -\!6 \ -\!5 \ -\!4 \ -\!3 \ -\!2]$, while the analogue of the strategic pile for $\pi$ is $\{6,-1,4\}$.
\end{example}

Next we derive a few parity results for \cdr analogously to corresponding results for \cds.

\begin{corollary}\label{cdrcds}
Let $\pi \in \s_n^{\pm}$ be satisfy $\pi(j)\bmod 2 = 0 \Leftrightarrow j\bmod 2 = 1$ for all $j \le n$. Then $\pi$ in not \cdr-sortable.
\end{corollary}

\begin{proof}
Let $\varphi$ be a fixed point resulting from applying \cdr operations. If all entries in $\varphi$ are negative, then Theorem~\ref{cdrsortable} implies that $\pi$ is not \cdr-sortable. On the other hand, if $\varphi\in\s_n$, then Lemma~\ref{cdrparityinvariance} and the parity property of $\pi$ imply that $\varphi$ also has this parity property. By Lemma~\ref{parityswitch}, we must have that $n$ is even and $\varphi$ is not \cds-sortable. Then by Theorem~\ref{swapsortable}, the strategic pile of $\varphi$ is nonempty, and so $0$ and $2n$ appear in the same cycle of $D_{\pi}$. Thus, by Theorem~\ref{cdrsortability}, the signed permutation $\pi$ is not \cdr-sortable.
\end{proof}

Note that a signed permutation satisfying the parity property hypothesis of Corollary~\ref{cdrcds} may still be reverse \cdr-sortable, as illustrated by $[8 \ 3 \ 6 \ 1 \ -\!4 \ 7 \ 2 \ 5]$. 

\begin{lemma}[\cdr Parity Invariance]\label{cdrparityinvariance}
Let $\pi$ be a signed permutation with either of the following properties:
\begin{enumerate}
\item $\pi(j) \bmod 2=0 \iff j \bmod 2=0$ for all $j$, or
\item $\pi(j) \bmod 2=1 \iff j \bmod 2=0$ for all $j$.
\end{enumerate}
Let $\sigma$ be obtained by applying \cdr to $\pi$ with some context. Then $\sigma$ satisfies the same property above.
\end{lemma}

\begin{proof}
There is nothing to prove if $\pi$ is a \cdr fixed point, so assume that it is not. Then for some entry $x$ of $\pi$, both $x$ and $-(x+1)$ appear in $\pi$. Without loss of generality, we can assume that (for some entries $a$ and $b$) $\pi$ has the following form for $x > 0$:
$$\pi = [ \cdots \ x \ a \ \cdots \ b \ -\!(x+1) \ \cdots],$$
and the \cdr operation with context $\pm(x,x+1)$ will transform the segment $a \ \cdots \ b \ -\!(x+1)$ of $\pi$ into $(x+1) \ -\!b \ \cdots \ -\!a$ to produce $\sigma$. Certainly $x$ and $x+1$ have different parities. Moreover, if $\pi$ satisfies one of the parity properties listed above, then $x$ and $a$ have different parities, as do $b$ and $x+1$. Thus, because for any entry $z$, $z$ and $-z$ always have the same parity, the signed permutation $\sigma$ satisfies the same parity property that $\pi$ did.
\end{proof}

\section{\cdr and Games}\label{section:cdrgames}

As already noted, strategy plays a crucial role in sorting signed permutations using \cdr. This suggests natural combinatorial games based on \cdr. We define one class of such games: Fix a signed permutation $\pi$ and a set  \textsf{F} of \cdr fixed points of $\pi$. The two-person game
$\textsf{CDR}(\pi,\textsf{F})$
is played as follows.

\begin{quote}
Player ONE applies a \cdr operation to $\pi$ to produce $\pi_1$. Player TWO applies a \cdr operation to $\pi_1$ to produce $\pi_2$. The players alternate \cdr moves in this manner until a fixed point $\varphi$ is reached. Player ONE wins if $\varphi \in \textsf{F}$. Otherwise, player TWO wins.
\end{quote}

The \emph{normal play} and \emph{misere} versions of the game are defined as follows:

\begin{definition}\
\begin{enumerate}\renewcommand{\labelenumi}{(\alph{enumi})}
\item The normal play game game $\textsf{NCDR}(\pi)$ proceeds as $\textsf{CDR}(\pi,\textsf{F})$, but the player that makes the last legal move wins.
\item The misere play game $\textsf{MCDR}(\pi)$ proceeds as $\textsf{CDR}(\pi,\textsf{F})$, but the player that makes the last legal move loses.
\end{enumerate}
\end{definition}

As all these games are of finite length, and none of them ends in a draw, Zermelo's Theorem~\ref{fundthfingames} implies that for each choice of $\pi\in \s_n^{\pm}$ some player has a winning strategy in the game. This fact suggests the following three decision problems.

\begin{quote}
\framebox{
\parbox{.8\textwidth}{
{\bf D.6 \hspace{0.1in} NORMAL SORTING GAME:}

INSTANCE: A positive integer $n$, and a signed permutation $\pi$.

QUESTION: Does ONE have a winning strategy in the game $\textsf{NCDR}(\pi)$?}}
\end{quote}

\begin{quote}
\framebox{
\parbox{.8\textwidth}{
{\bf D.7 \hspace{0.1in} MISERE SORTING GAME:}

INSTANCE: A positive integer $n$ and a signed permutation $\pi$.

QUESTION: Does ONE have a winning strategy in the game $\textsf{CDRM}(\pi)$?}}
\end{quote}

\begin{quote}
\framebox{
\parbox{.8\textwidth}{
{\bf D.8 \hspace{0.1in} FIXED POINT SORTING GAME:}

INSTANCE: A positive integer $n$, a signed permutation $\pi$, and a set \textsf{F} of \cdr-fixed points.

QUESTION: Does ONE have a winning strategy in the game $\textsf{CDR}(\pi,\textsf{F})$?}}
\end{quote}

Since we originally formulated these games, the normal play and misere sorting game decision problems, \textbf{D.6} and \textbf{D.7}, have been solved in \cite{JSST2} by finding a linear time decision process. The theory of the fixed point sorting game for signed permutations, however, is not well understood. One of the main obstacles is that the set of \cdr fixed points reachable through the applications of \cdr operations has not been characterized in terms of simple structures, as is done in Theorem \ref{strategicpilecorollary} for the case of \cds.

\section{An application}\label{section:application}

The operations \cdr and \cds are the hypothesized operations for decrypting the micronuclear precursors of macronuclear genes in ciliates. One would expect that these operations should be robust and error free. The \cds Inevitability Theorem, Theorem \ref{cdsinevitable}, shows that for \cds-sortable micronuclear precursors, applying \cds operations in any order would successfully decrypt the micronuclear precursor. Thus, no strategic intervention is required. However, as illustrated in Example \ref{ex:cdrlengths}, for \cdr-sortable micronuclear precursors, success in decrypting the precursor depends on the order in which \cdr operations are applied. This raises the question of whether there is a mechanism that strategically governs the order in which \cdr operations are applied by the ciliate decryptome. In \cite{JSST2} this question has been answered by proving that the current mathematical model is robust enough that no additional strategic mechanisms are needed.

\section{Future work}\label{section:future}

The three most basic decision problems left open by this investigation are the \cdr FIXED POINT decision problem, {\bf D.5} and the FIXED POINT SORTING GAME decision problems {\bf D.2} and {\bf D.8}. We do not expect {\bf D.5} to be an NP-complete problem. Experience suggests a higher complexity level for decision problems {\bf D.2} and {\bf D.8} - it would be interesting to know whether these two decision problems are PSPACE complete.

There are also several enumerative questions that are not yet fully understood. For example, how many \cds-sortable permutations are in $\s_n$? The values for the first few $n$ are given in Table~\ref{table:enumeration} and appear in \cite{oeis} as A249165, but a general formula is not known.

\begin{table}[htbp]
$$\begin{array}{c|c|c|c|c|c|c|c|c|c|c}
n & 1 & 2 & 3 & 4 & 5 & 6 & 7 & 8 & 9 & 10\\
\hline
\# \text{\cds-sortable } \pi & 1 & 1 & 4 & 13 & 72 & 390 & 2880 & 21672 & 201600 & 1935360
\end{array}$$
\caption{The number of \cds-sortable permutations in $\s_n$ for $n \le 10$.}\label{table:enumeration}
\end{table}

E.~Holmes and P.~A.~Plummer have independently arrived at and extended this sequence of data \cite{HoPl}, and make the following conjecture.

\begin{conj}[Holmes and Plummer]\label{hpconjecture}
$\# \{\text{\cds-sortable elements in } \s_{2k+1}\} = (k+1) (2k)!.$
\end{conj}

We know of no corresponding conjecture for enumerating \cds-sortable elements in $\s_{2k}$.  
Enumerating the number of elements of $\s_n^{\pm}$ that are \cdr-sortable has been considered in \cite{LRSSS}, and appears as sequence A260511 in \cite{oeis}. In \cite{LRSSS} it is shown that this sequence enumerates the set of elements of $\s_n^{\pm}$ reverse sortable by \cdr. As a result the number of elements of $\s_n^{\pm}$ not sortable or reverse sortable by applications of \cdr can be obtained from the sequence A260511.

The \cds-sortability criterion of Theorem~\ref{swapsortable} suggests that there might be a characterization in terms of permutation patterns. Such a relationship would not be entirely surprising since sorting operations and permutation patterns have been linked before: Stack-sortable permutations can be characterized by avoidance of a single pattern \cite{knuth}. Though \cds-sortable permutations are not obviously characterized in this manner (since they are not closed under (classical) pattern containment), there may be some other way to describe these objects in terms of patterns, perhaps utilizing the various presentations of a permutation. 

In regards to the mathematical model for ciliate micronuclear decryption arising from the hypotheses in \cite{EPModel}: In \cite{EHPPR} and some earlier papers it is assumed that besides \cdr and \cds, there is an additional sorting operation, \emph{boundary} \textsf{ld}, that would sort a \cds fixed point to the identity permutation. This operation corresponds to Christie's minimal block interchange applied to a \cds fixed point. At this stage it has not been experimentally confirmed whether this operation occurs, nor has a satisfactory molecular mechanism for this operation in combination with \cdr and \cds been described. Prior to \cite{Cetal}, all known examples of permutations representing micronuclear precursors of macronuclear genes in ciliates were \cds-sortable. However, the permutation representing the red micronuclear precursor in Figure 3 B,  and the permutation representing the gold micronuclear precursor in Figure 3 C of \cite{Cetal} are both \cds fixed points, and an operation beyond \cds or \cdr is needed to sort this pattern to the identity permutation corresponding to the macronuclear gene. It would be interesting to learn how these two particular micronuclear precursors are in fact processed by the ciliate decryptome.

\section{Acknowledgements}
The research represented in this paper was funded by an NSF REU grant DMS 1062857, by Boise State University and by the Department of Mathematics at Boise State University. Wauck was a member of the 2012 REU cohort while Adamyk, Holmes, Moritz and Mayfield were members of the 2013 REU cohort.

We thank E.~Holmes and P.~A.~Plummer for valuable communication regarding their findings, including Conjecture~\ref{hpconjecture}, and confirmation of the data appearing in Table~\ref{table:enumeration}.

\end{document}